\documentclass[11pt,twoside,reqno]{amsart}
\usepackage[utf8]{inputenc}
\usepackage[T1]{fontenc}
\usepackage{import}
\usepackage{newclude}

\usepackage{anysize}
\usepackage{amsmath,amsthm,amscd,amssymb}
\usepackage{mathtools}
\usepackage{stmaryrd}
\usepackage{wasysym}
\usepackage{mathrsfs}
\usepackage{mathdots}
\usepackage{dsfont}
\usepackage{bm}
\usepackage{scalerel,stackengine}

\usepackage{xcolor}
\usepackage{graphicx}
\usepackage{enumitem}
\usepackage{soul}
\usepackage{float}
\usepackage{makecell}
\usepackage{setspace}
\usepackage{comment}
\usepackage{csquotes}
\usepackage[all]{xy}
\usepackage{tikz}
\usepackage{tikz-cd}

\usepackage[myheadings]{fullpage}

\usepackage[pagebackref=true,hypertexnames=false]{hyperref}
\hypersetup{colorlinks=true,linkcolor=black,citecolor=black}
\renewcommand*{\backref}[1]{}
\renewcommand*{\backrefalt}[4]{({\tiny%
   \ifcase #1 Not cited.%
         \or Cited on page~#2.%
         \else Cited on pages #2.%
   \fi%
   })}
\usepackage{aliascnt}
\usepackage{cleveref}

\DeclareFontFamily{U}{mathb}{}

\DeclareFontShape{U}{mathb}{m}{n}{
   <-5.5>  mathb5
   <5.5-6.5> mathb6
   <6.5-7.5> mathb7
   <7.5-8.5> mathb8
   <8.5-9.5> mathb9
   <9.5-11>  mathb10
   <11->     mathb12
}{}

\DeclareSymbolFont{mathb}{U}{mathb}{m}{n}

\DeclareMathSymbol{\lefttorightarrow}{3}{mathb}{"FC}
\DeclareMathSymbol{\righttoleftarrow}{3}{mathb}{"FD}

\makeatletter
\newcommand{\acts}{%
  \mathrel{\mathpalette\@acts\righttoleftarrow}}
\newcommand{\actedby}{%
  \mathrel{\mathpalette\@acts\lefttorightarrow}}
\newcommand{\@acts}[2]{\reflectbox{$\m@th#1#2$}}
\makeatother

\numberwithin{equation}{section}
\newcommand\mtop{.95in}
\newcommand\mbottom{.95in}
\newcommand\mleft{1in}
\newcommand\mright{1in}
\usepackage[top = \mtop, bottom = \mbottom, left = \mleft, right=\mright]{geometry}

\DeclareMathOperator{\Mat}{Mat}

\newtheorem{thm}{Theorem}[section]
\newaliascnt{theorem}{thm}
\newtheorem{theorem}[theorem]{Theorem}
\aliascntresetthe{theorem}
\newaliascnt{example}{thm}

\aliascntresetthe{example}
\newaliascnt{prop}{thm}
\newtheorem{prop}[prop]{Proposition}
\aliascntresetthe{prop}
\newaliascnt{lemma}{thm}
\newtheorem{lemma}[lemma]{Lemma}
\aliascntresetthe{lemma}
\newaliascnt{conj}{thm}
\newtheorem{conj}[conj]{Conjecture}
\aliascntresetthe{conj}
\newaliascnt{cor}{thm}
\newtheorem{cor}[cor]{Corollary}
\aliascntresetthe{cor}
\theoremstyle{definition}
\newaliascnt{defi}{thm}
\newtheorem{defi}[defi]{Definition}
\aliascntresetthe{defi}
\newaliascnt{notation}{thm}

\aliascntresetthe{notation}
\newaliascnt{rmk}{thm}

\aliascntresetthe{rmk}
\newaliascnt{remark}{thm}
\newtheorem{remark}[remark]{Remark}
\aliascntresetthe{remark}

\crefname{thm}{theorem}{theorems}
\Crefname{thm}{Theorem}{Theorems}
\crefname{theorem}{theorem}{theorems}
\Crefname{theorem}{Theorem}{Theorems}
\crefname{prop}{proposition}{propositions}
\Crefname{prop}{Proposition}{Propositions}
\crefname{lemma}{lemma}{lemmas}
\Crefname{lemma}{Lemma}{Lemmas}
\crefname{conj}{conjecture}{conjectures}
\Crefname{conj}{Conjecture}{Conjectures}
\crefname{cor}{corollary}{corollaries}
\Crefname{cor}{Corollary}{Corollaries}
\crefname{defi}{definition}{definitions}
\Crefname{defi}{Definition}{Definitions}
\crefname{notation}{notation}{notations}
\Crefname{notation}{Notation}{Notations}
\crefname{rmk}{remark}{remarks}
\Crefname{rmk}{Remark}{Remarks}
\crefname{remark}{remark}{remarks}
\Crefname{remark}{Remark}{Remarks}

\newcommand\reallywidehat[1]{%
\savestack{\tmpbox}{\stretchto{%
  \scaleto{%
    \scalerel*[\widthof{\ensuremath{#1}}]{\kern-.6pt\bigwedge\kern-.6pt}%
    {\rule[-\textheight/2]{1ex}{\textheight}}
  }{\textheight}%
}{0.5ex}}%
\stackon[1pt]{#1}{\tmpbox}%
}
\DeclareSymbolFont{bbold}{U}{bbold}{m}{n}
\DeclareSymbolFontAlphabet{\mathbbold}{bbold}

\makeatletter
\def\@tocline#1#2#3#4#5#6#7{\relax
  \ifnum #1>\c@tocdepth 
  \else
    \par \addpenalty\@secpenalty\addvspace{#2}%
    \begingroup \hyphenpenalty\@M
    \@ifempty{#4}{%
      \@tempdima\csname r@tocindent\number#1\endcsname\relax
    }{%
      \@tempdima#4\relax
    }%
    \parindent\z@ \leftskip#3\relax \advance\leftskip\@tempdima\relax
    \rightskip\@pnumwidth plus4em \parfillskip-\@pnumwidth
    #5\leavevmode\hskip-\@tempdima
      \ifcase #1
       \or\or \hskip 1em \or \hskip 2em \else \hskip 3em \fi%
      #6\nobreak\relax
    \hfill\hbox to\@pnumwidth{\@tocpagenum{#7}}\par
    \nobreak
    \endgroup
  \fi}
\makeatother

\newcommand{\Z}{\mathbb{Z}}

\newcommand{\F}{\mathbb{F}}

\newcommand{\E}{\mathbb{E}}

\renewcommand{\l}{\lambda}

\renewcommand{\P}{\mathbf{P}}

\newcommand{\Y}{\mathbb{Y}}

\DeclareMathOperator{\Hom}{Hom}
\DeclareMathOperator{\Sur}{Sur}

\DeclareMathOperator{\Cok}{Cok}

\DeclareMathOperator{\Ext}{Ext}

\DeclareMathOperator{\Aut}{Aut}

\DeclareMathOperator{\Span}{span}

\DeclareMathOperator{\GL}{GL}

\title[Sharp Threshold for Universality of Rational Canonical Forms]{Sharp Threshold for Universality of Rational Canonical Forms over a Finite Field}
\author{Jiahe Shen}

\date{\today}

\begin{document}

\thanks{I thank Roger Van Peski for many helpful discussions and suggestions, and Jason Fulman for helpful comments on the draft and for pointing me to relevant literature on random matrices over finite fields. This work is supported by NSF grant DMS-2246576 and Simons Investigator grant 929852.}

\begin{abstract}
We study the rational canonical form of sparse random matrices over a finite field. Suppose $A_n\in\Mat_n(\F_p)$ has independent and $\alpha_n$-balanced entries. We prove that if
$$
n\alpha_n-\log n\longrightarrow+\infty,
$$
then, for every fixed collection of distinct monic irreducible polynomials over $\F_p$, the corresponding primary partitions of $A_n$ converge jointly to the same asymptotically independent Cohen-Lenstra distributions as in the uniform random matrix model. This strengthens the finite-field cokernel universality theorem of Lee \cite[Theorem 1.3]{lee2025sharp}, which assumes $\alpha_n=c\log n/n$ for a fixed $c>1$, by resolving the additive endpoint case explicitly left open in his work \cite[p.~3]{lee2025sharp}. Moreover, our result extends universality from ordinary cokernels to the full rational canonical form. The condition is optimal for universality uniformly over fixed irreducible polynomials.

Our proof applies the surjection moment method developed by Sawin-Wood \cite{sawin2022moment} over $\mathbb{F}_p[t]$, together with a refinement of the Fourier estimates of Lee \cite{lee2025sharp} that avoids losing a fixed proportion of the leading decay. We also construct degree-$d$ critical sparse obstructions, suggesting a polynomial-dependent threshold of $1/d$ for irreducible polynomials of minimal degree $d$. We also construct degree-$d$ critical sparse obstructions, suggesting a polynomial-dependent threshold of $1/d$ for irreducible polynomials of minimal degree $d$.
\end{abstract}

\maketitle



\tableofcontents

\section{Introduction}\label{sec: Introduction}

\subsection{Main results}

Let $p$ be a prime number, and let $\F_p$ denote the finite field of order $p$. We write $\Mat_n(\F_p)$ for the set of $n\times n$ matrices over $\F_p$. This paper studies the limiting rational canonical form of random matrices $A_n\in\Mat_n(\F_p)$ with independent entries whose distributions may become increasingly sparse as $n\to\infty$.

We recall the primary partitions associated with the rational canonical form. Let $\Y$ be the set of integer partitions, including the zero partition, and enumerate the monic irreducible nonconstant polynomials in $\F_p[t]$ by $F_1,F_2,\ldots$. For a matrix $A\in\Mat_n(\F_p)$, regard $\F_p^n$ as an $\F_p[t]$-module by letting $t$ act through $A$. For every $F_i$, the $F_i$-primary component has a unique decomposition of the form
\begin{equation}\label{eq:intro_primary_decomposition}
    \bigoplus_{j\ge1}\F_p[t]/(F_i^{\lambda_j^{(F_i)}}),
\end{equation}
where $\lambda^{(F_i)}=\lambda^{(F_i)}(A)\in\Y$ is a partition, taken to be the zero partition when the primary component is trivial. Equivalently, the nonzero parts of $\lambda^{(F_i)}(A)$ specify the companion blocks $C(F_i^{\lambda_j^{(F_i)}})$ in the primary rational canonical form of $A$. Thus the collection of these partitions determines the conjugacy class of $A$ over $\F_p$.

For $\alpha>0$, a random variable $\xi\in\F_p$ is called $\alpha$-balanced if $\P(\xi=a)\le1-\alpha$ for every $a\in\F_p$. A random matrix is $\alpha$-balanced if its entries are independent and $\alpha$-balanced; they need not be identically distributed.

The sparsity threshold for universality is already subtle for the ordinary cokernel. Lee~\cite[Theorem 1.3]{lee2025sharp} proved that its limiting distribution agrees with that of a uniform random matrix when $\alpha_n=c\log n/n$ for any fixed $c>1$. Lee~\cite[Introduction, pp.~2--3]{lee2025sharp} also constructed nonuniversal examples at $\alpha_n=(\log n+C)/n$ for fixed $C\in\mathbb R$ and explicitly left open whether $n\alpha_n-\log n\to+\infty$ suffices for universality. This additive endpoint condition allows $\alpha_n=(\log n+w_n)/n$ with $w_n\to\infty$ arbitrarily slowly, and does not follow from the fixed-$c$ result. Our main theorem answers this question affirmatively and establishes joint universality for the primary partitions associated with any fixed collection of irreducible polynomials.

\begin{theorem}\label{thm:main}
Let $(\alpha_n)_{n\ge1}$ be a sequence of positive real numbers satisfying
\begin{equation}\label{eq:alpha_condition}
    n\alpha_n-\log n\longrightarrow+\infty.
\end{equation}
For each $n$, let $A_n\in\Mat_n(\F_p)$ be an $\alpha_n$-balanced random matrix. For any fixed tuple of distinct monic irreducible polynomials $F_{i_1},\ldots,F_{i_r}$, the joint distribution of $(\lambda^{(F_{i_1})}(A_n),\ldots,\lambda^{(F_{i_r})}(A_n))$ converges pointwise to the product measure $\prod_{j=1}^r\mu_{F_{i_j}}$ on $\Y^r$ as $n\to\infty$, where
\begin{equation}\label{eq:mu_F}
    \mu_{F_i}(\lambda)=\frac{1}{\#\Aut_{F_i}(\lambda)}\prod_{\ell=1}^{\infty}(1-p^{-\ell\deg F_i}),\qquad \lambda\in\Y.
\end{equation}
Here $\Aut_{F_i}(\lambda)$ denotes the automorphism group of the $\F_p[t]$-module of type $\lambda$ in \eqref{eq:intro_primary_decomposition}.
\end{theorem}

\begin{remark}\label{rem:fulman}
Fulman~\cite[Section 3]{fulman1997probability} studies uniform matrices in $\GL_n(\F_p)$ rather than $\Mat_n(\F_p)$. The $t$-primary component is therefore trivial in that model, while for each irreducible polynomial $F_i\ne t$ the limiting measure agrees with $\mu_{F_i}$ in \Cref{thm:main}.
\end{remark}

The measure $\mu_{F_i}$ is a Cohen-Lenstra type distribution associated with the residue field $\F_p[t]/(F_i)$, of cardinality $p^{\deg F_i}$. In particular, \Cref{thm:main} determines both the occurrence of a prescribed irreducible factor in the characteristic polynomial and the sizes of its primary rational canonical blocks. It also shows that the corresponding partitions are asymptotically independent for distinct fixed irreducible polynomials. When $F_i=t$, the partition $\lambda^{(t)}(A_n)$ records the sizes of the nilpotent Jordan blocks, and $\ell(\lambda^{(t)}(A_n))=\dim_{\F_p}\Cok(A_n)$. Thus \Cref{thm:main} contains the additive endpoint universality of the ordinary cokernel as a special case.

The condition in \eqref{eq:alpha_condition} is optimal for universality uniformly over fixed irreducible polynomials: the sparse counterexamples of Lee~\cite[Introduction, pp.~2--3]{lee2025sharp} show that $\alpha_n=(\log n+C)/n$ with fixed $C$ does not suffice even for the $t$-primary component. We return to this obstruction and its higher-degree analogues in \Cref{subsec:critical_sparse}.

As a consequence of \Cref{thm:main}, we also obtain the limiting multiplicities of prescribed irreducible factors in the characteristic polynomial. Write $e_{F_i}=e_{F_i}(A_n):=|\lambda^{(F_i)}(A_n)|$, so that $e_{F_i}$ is the multiplicity of $F_i$ in $\det(tI_n-A_n)$.

\begin{cor}\label{cor:multiplicity}
Under the assumptions of \Cref{thm:main}, the joint law of $(e_{F_{i_1}},\ldots,e_{F_{i_r}})$ converges to the product measure $\prod_{j=1}^r\nu_{F_{i_j}}$, where
\begin{equation}\label{eq:nu_F}
    \nu_{F_i}(m)=p^{-m\deg F_i}\prod_{\ell\ge m+1}(1-p^{-\ell\deg F_i}),\qquad m\in\mathbb Z_{\ge0}.
\end{equation}
\end{cor}

\subsection{Relation to previous work and proof strategy}\label{subsec:previous_work}

The primary partitions of uniform random matrices over finite fields have been studied through generating functions, cycle indices, and symmetric functions. Fulman~\cite[Section 3]{fulman1997probability} initiated a systematic study of their limiting distributions, and further developments appear in Fulman~\cite{fulman1999probabilistic,fulman2002random,fulman2013cohen}. For matrices with nonuniform entries, Charlap-Rees-Robbins~\cite{charlap1990asymptotic} obtained universality for invertibility under a fixed nondegeneracy assumption, while Bl\"omer-Karp-Welzl~\cite{blomer1997rank} and Cooper~\cite{cooper2000distribution} studied rank and nonsingularity in sparse finite-field models. Luh-Meehan-Nguyen~\cite[Theorem 1.7]{luh2021some} obtained universality results for the occurrence of prescribed irreducible factors in the characteristic polynomial. Our result determines the associated primary partitions, rather than only their presence or absence, in a sparse regime where the ordinary cokernel had not previously been understood.

The study of random matrix cokernels and Cohen-Lenstra distributions goes back to Friedman-Washington~\cite{friedman1989distribution}. Wood~\cite{wood2019random} and Nguyen-Wood~\cite[Theorem 4.1]{nguyen2022random} established universality results for broad classes of independent nonuniform entries, and Sawin-Wood~\cite[Theorem 1.6 and Lemma 6.3]{sawin2022moment} developed a general method for recovering module distributions from surjection moments. In the sparse finite-field setting, Lee~\cite[Theorem 1.3]{lee2025sharp} established the sharp multiplicative threshold $c>1$ for ordinary cokernels. Jung-Lee-Yu~\cite[Theorem 1.5]{jungleeYu2026sharp} subsequently obtained related sharp-threshold results for nonsymmetric, symmetric, and alternating random matrix models over the $p$-adic integers. The present paper resolves the additive endpoint left open by Lee~\cite[Introduction, p.~3]{lee2025sharp} and extends the resulting universality statement to fixed primary partitions.

The algebraic starting point is the finite $\F_p[t]$-module
\begin{equation}\label{eq:function_field_cokernel}
    X_n:=\Cok_{\F_p[t]}(tI_n-A_n).
\end{equation}
As explained in \Cref{sec: Proof of the main results}, this module is naturally isomorphic to $\F_p^n$ with $t$ acting through $A_n$. In particular, its $F_i$-primary component has the decomposition
\begin{equation}\label{eq:primary_cokernel}
    X_n[F_i^\infty]\cong\bigoplus_{j\ge1}\F_p[t]/(F_i^{\lambda_j^{(F_i)}(A_n)}).
\end{equation}
Related distributions of polynomial cokernels of random integral matrices were studied by Cheong-Yu~\cite{cheong2023distribution}. The main intermediate result in our proof is the convergence
\begin{equation}\label{eq:intro_moment_convergence}
    \E\#\Sur_{\F_p[t]}(X_n,G)\longrightarrow1
\end{equation}
for every fixed finite $\F_p[t]$-module $G$. The moment-determination theorem of Sawin-Wood~\cite[Theorem 1.6 and Lemma 6.3]{sawin2022moment} then yields the joint limiting distributions in \Cref{thm:main}.

The main difficulty is establishing these surjection moments at the additive endpoint. Following Lee~\cite[Lemma 2.1 and Sections 2.2--2.4]{lee2025sharp}, we reduce to two-point entry distributions and organize the Fourier sums according to support sizes and popular values. The first two large-support cases are handled essentially by the arguments of Lee~\cite[Section 2.3]{lee2025sharp}. The difficulty arises in the fully orthogonal large-support case and in the small-support range: the estimates of Lee~\cite[Theorem 2.2 and Proposition 2.10]{lee2025sharp} lose a fixed proportion of the leading Fourier decay. Such a loss can be absorbed when $\alpha_n=c\log n/n$ for fixed $c>1$, but not when $n\alpha_n-\log n$ diverges arbitrarily slowly.

Our main analytic input is an averaged Fourier damping estimate based on convexity, combined with the two-block decomposition of Lee~\cite[Section 2.4]{lee2025sharp} in the fully orthogonal case. For sufficiently small exceptional supports of size $s$, our Fourier estimate gives a decay of the form $\exp(-s(n\alpha_n-O(1)))$, retaining the full leading coefficient of $n\alpha_n$. As shown in the proofs of \Cref{endpoint:lem:small,endpoint:lem:doublesum}, the corresponding support sums are bounded by $O(ne^{-n\alpha_n})=o(1)$, while larger exceptional supports are handled by entropy estimates. This refinement of the Fourier analysis of Lee~\cite[Sections 2.2--2.4]{lee2025sharp} is what allows us to prove \eqref{eq:intro_moment_convergence} under the sharp additive condition \eqref{eq:alpha_condition}; see \Cref{subsec:sharp-function-field-moments}.

\subsection{Uniform and polynomial-dependent thresholds}\label{subsec:degree_threshold}

\Cref{thm:main} gives a threshold valid simultaneously for every fixed collection of irreducible polynomials. If only certain polynomials are prescribed, their degrees suggest a lower threshold. Fix distinct irreducible polynomials $F_{i_1},\ldots,F_{i_r}$ and set $d:=\min_{1\le j\le r}\deg F_{i_j}$.

\begin{conj}\label{conj:degree_threshold}
Let $A_n\in\Mat_n(\F_p)$ be an $\alpha_n$-balanced random matrix with independent entries. If
\begin{equation}\label{eq:conjectural_threshold}
    nd\alpha_n-\log n\longrightarrow+\infty,
\end{equation}
then the joint distribution of $(\lambda^{(F_{i_1})}(A_n),\ldots,\lambda^{(F_{i_r})}(A_n))$ converges pointwise to the same product of Cohen-Lenstra measures as in \Cref{thm:main}.
\end{conj}

The coefficient $1/d$ is suggested by a degree-$d$ analogue of the critical example of Lee~\cite[Introduction, pp.~2--3]{lee2025sharp}, which also shows that this coefficient cannot be decreased in a general sufficient condition. Fix an irreducible polynomial $F$ of degree $d$ and start with a deterministic matrix containing approximately $n/d$ companion blocks $C(F)$. Independently perturb each entry with probability $\alpha_n$, leaving it unchanged otherwise. A given block survives as a quotient whenever all entries in its $d$ rows remain unchanged, an event of probability $(1-\alpha_n)^{dn}$. At $\alpha_n=(\log n+C)/(dn)$, where $C\in\mathbb R$ is fixed, the number of surviving blocks converges in distribution to $\operatorname{Pois}(e^{-C}/d)$. These blocks give a lower bound on the number of parts in the $F$-primary partition that is incompatible with its Cohen-Lenstra limiting distribution. The construction is given in \Cref{subsec:critical_sparse}. It proves the necessity of the conjectured threshold, but not its sufficiency.

A possible route to \Cref{conj:degree_threshold} would be to establish \eqref{eq:intro_moment_convergence} for finite target modules $G$ supported only at $F_{i_1},\ldots,F_{i_r}$ under the weaker condition \eqref{eq:conjectural_threshold}. This appears to require Fourier estimates sensitive to the prescribed irreducible factors, beyond the uniform bounds used to prove \Cref{thm:main}.

\subsection{Organization of the paper}

In \Cref{sec: Preliminaries}, we recall the structure of finite $\F_p[t]$-modules. In \Cref{sec: Proof of the main results}, we identify the primary components of $\Cok_{\F_p[t]}(tI_n-A_n)$ with the rational canonical partitions, prove the endpoint surjection-moment estimate by refining the Fourier analysis of Lee~\cite[Sections 2.2--2.4]{lee2025sharp}, and deduce \Cref{thm:main,cor:multiplicity} using the method of Sawin-Wood~\cite[Theorem 1.6 and Lemma 6.3]{sawin2022moment}. We then construct the degree-one and degree-$d$ critical sparse examples.

\subsection{AI statement}

The research questions addressed in this paper were formulated by the author. Through an iterative dialogue with the author, OpenAI's ChatGPT generated the initial proof of the endpoint Fourier estimate, which led to an extension of the universality theorem of Lee~\cite[Theorem 1.3]{lee2025sharp} to the additive critical window. The author critically evaluated, corrected, and refined the model's suggestions, verified the mathematical arguments, and wrote the final manuscript. ChatGPT was also used to assist with exposition and proofreading. The author takes full responsibility for the correctness of the results.

\section{Preliminaries}\label{sec: Preliminaries}

In this section, we review the basic structure of finite $\F_p[t]$-modules and record the algebraic facts needed for the proof of our main results.

\begin{defi}
Recall from \Cref{sec: Introduction} that 
$$\Y=\{\l=(\l_1,\l_2,\ldots):\l_1\ge\l_2\ge\ldots\ge0,\l_i\in\Z,\l_i=0 \text{ for all but finitely many }i\}.$$ 
The integers $\l_i>0$ are called the \emph{parts} of $\l$. Let $|\l| := \sum_{i\ge 1}\l_i,n(\lambda):=\sum_{i\ge 1} (i-1)\lambda_i$, and $m_k(\l):= \#\{i\mid \l_i = k\}$. 
\end{defi}

Now we can classify the isomorphism classes of finite $\F_p[t]$-modules. Following \Cref{sec: Introduction}, we write $S\subset\F_p[t]$ for the subset of monic, irreducible, nonconstant polynomials. Since $\F_p[t]$ is a PID, any finite $\F_p[t]$-module is a finite torsion module and hence
decomposes uniquely as a direct sum of its $f$-primary components (for $f\in S$).
The following proposition records the resulting classification of isomorphism classes.

\begin{prop}\label{prop: classification of finite module}
Every finite $\F_p[t]$-module $G$ admits a unique decomposition
\begin{equation}\label{eq: decomposition of f-submodules}
G\cong \bigoplus_{f\in S} G_f,
\end{equation}
where only finitely many summands are nonzero, and $G_f$ denotes the $f$-primary submodule of $G$.

Moreover, for each $f\in S$, there exists a unique partition
\[
\l^{(f)}=(\l^{(f)}_1,\l^{(f)}_2,\ldots)\in\Y
\]
such that
\[
G_f\cong \bigoplus_{j\ge 1}\F_p[t]/(f^{\l^{(f)}_j})
\]
as $\F_p[t]$-modules. Equivalently, if $f_1,\dots,f_r\in S$ are the distinct irreducible polynomials such that
$G_{f_i}\neq 0$, then
\[
G\cong \bigoplus_{i=1}^r\bigoplus_{j\ge 1}\F_p[t]/(f_i^{\l^{(f_i)}_j}),
\]
and the collection of partitions
\[
\l^{(f_1)},\dots,\l^{(f_r)}
\]
is uniquely determined by the $\F_p[t]$-module isomorphism class of $G$.
\end{prop}

In general, for $f\in S$, we say that the finite $\F_p[t]$-module $G$ has \emph{type} $\l$ at $f$ if the $f$-primary component $G_f$ in the decomposition \eqref{eq: decomposition of f-submodules} is isomorphic to $\left(\F_p[t]/(f^{\l_1})\right)\bigoplus\cdots\bigoplus\left(\F_p[t]/(f^{\l_j})\right)$. The following proposition studies the group of automorphisms of $\F_p[t]$-modules and counts their cardinality.

\begin{prop}\label{prop: decomposition and explicit expression of Aut}
We have
\begin{equation}\label{eq: decomposition of automorphism}
\Aut\left(\bigoplus_{i=1}^k\bigoplus_{j\ge 1}(\F_p[t]/(f_i^{\l^{(f_i)}_j}))\right)=\prod_{i=1}^k\Aut_{f_i}(\lambda^{(f_i)}),
\end{equation}
where for all $1\le i\le k$, $\Aut_{f_i}(\lambda^{(f_i)})$ is the group of automorphisms of $\bigoplus_{j\ge 1}(\F_p[t]/(f_i^{\l^{(f_i)}_j}))$, and
\begin{equation}\label{eq: explicit expression of automorphism}
\#\Aut_{f_i}(\lambda^{(f_i)})=p^{\deg f_i\cdot(|\l^{(f_i)}|+2n(\l^{(f_i)}))}\prod_{j\ge 1}(p^{-\deg f_i};p^{-\deg f_i})_{m_j(\l^{(f_i)})}.
\end{equation}
Here, the notation $(a;q)_m:=(1-a)(1-aq)\cdots(1-aq^{m-1}),m\ge 0$ refers to the $q$-Pochhammer symbol with the convention that $(a;q)_0=1$.
\end{prop}

\begin{proof}
The decomposition \eqref{eq: decomposition of automorphism} is valid because there are no non-trivial maps between the summands corresponding to each $i$. Also, notice that $\F_p[t]/(f_i)$ is the finite field with $p^{\deg f_i}$ elements. Therefore, the explicit expression \eqref{eq: explicit expression of automorphism} is given in \cite[Chapter II, Section 1]{macdonald1998symmetric}.
\end{proof}





\section{Proof of the main results}\label{sec: Proof of the main results}

We begin with the following standard interpretation of rational canonical form in terms of $\F_p[t]$-modules; see Hoffman and Kunze \cite[Section 7]{hoffmann1971linear}.

\begin{thm}\label{thm: rational canonical and smith normal form}
Let $A_n\in\Mat_n(\F_p)$, and $f\in S$. Then the type of $\Cok(tI_n-A_n)$ at $f$ is $\l^{(f)}=\l^{(f)}(A_n)$. Here, $\l^{(f)}(A_n)$ is the partition associated to the rational canonical form, as defined in \Cref{sec: Introduction}.
\end{thm}

\begin{proof}
Notice that the conjugation action of $\GL_n(\F_p)$ does not change the cokernel. Thus, there is no loss of generality in assuming that $A_n$ is already the rational canonical form, i.e.,
$$A_n=\begin{pmatrix}
R_1 & 0 & 0 & \cdots & 0 \\
0 & R_2 & 0 & \cdots & 0 \\
\vdots & \vdots & \vdots & \ddots & \vdots \\
0 & 0 & 0 & \cdots & R_k \\
\end{pmatrix}.$$
Here, each block matrix $R_i$ corresponds to different polynomials in $S$. Denote $\l^{(f)}=(\l_1,\ldots,\l_j)$. We furthermore assume that $R_1$ corresponds to $f$, so that the type of $\Cok(tI_n-A_n)$ at $f$ is the same as $\Cok(tI_{|\l|\deg f}-R_1)$, and
$$R_1=\begin{pmatrix}
C(f^{\lambda_1}) & 0 & 0 & \cdots & 0 \\
0 & C(f^{\lambda_2}) & 0 & \cdots & 0 \\
\vdots & \vdots & \vdots & \ddots & \vdots \\
0 & 0 & 0 & \cdots & C(f^{\lambda_j}) \\
\end{pmatrix}.$$
In this case, we have
$$\Cok(tI_{|\l|\deg f}-R_1)=\bigoplus_{i=1}^j\Cok(tI_{\l_i\deg f}-C(f^{\l_i}))=\bigoplus_{i=1}^j\left(\F_p[t]/(f^{\lambda_i})\right),$$
which ends the proof.
\end{proof}

The following lemma is a necessary preparation for our proof of \Cref{thm:main}.

\begin{lemma}\label{lem:ext-hom-correct}
Let $f\in S$, let $k\ge 1$, and set
$$
R_f:=\F_p[t]/(f^k), \qquad K:=\F_p[t]/(f).
$$
Let $M$ be a finite $R_f$-module whose type at $f$ has largest part at most $k-1$.
Equivalently, as an $R_f$-module,
$$
M\cong \bigoplus_{j=1}^m \F_p[t]/(f^{\lambda_j})
\qquad\text{with}\qquad 1\le \lambda_j\le k-1.
$$
Then
$$
\#\Ext^1_{R_f}(M,K)=\#\Hom_{R_f}(M,K).
$$
More precisely, both groups are $K$-vector spaces of dimension $m$.
\end{lemma}

\begin{proof}
Since both $\Hom_{R_f}(-,K)$ and $\Ext^1_{R_f}(-,K)$ commute with finite direct sums in the first variable, it suffices to treat the cyclic module
$$
M_\lambda:=R_f/(f^\lambda), \qquad 1\le \lambda\le k-1.
$$

For such $\lambda$, $M_\lambda$ admits the standard $2$-periodic free resolution
$$
\cdots \xrightarrow{\cdot f^{k-\lambda}} R_f
\xrightarrow{\cdot f^\lambda} R_f
\xrightarrow{\cdot f^{k-\lambda}} R_f
\xrightarrow{\cdot f^\lambda} R_f
\longrightarrow M_\lambda \longrightarrow 0.
$$
Applying $\Hom_{R_f}(-,K)$, we obtain the cochain complex
$$
0\to \Hom_{R_f}(R_f,K)\xrightarrow{0}\Hom_{R_f}(R_f,K)\xrightarrow{0}\Hom_{R_f}(R_f,K)\xrightarrow{0}\cdots
$$
because $f$ acts trivially on $K=R_f/(f)$, hence multiplication by $f^\lambda$ and by $f^{k-\lambda}$ induce the zero map on $\Hom_{R_f}(R_f,K)\cong K$.

Therefore,
$$
\Hom_{R_f}(M_\lambda,K)\cong K,
\qquad
\Ext^1_{R_f}(M_\lambda,K)\cong K.
$$
Now if
$$
M\cong \bigoplus_{j=1}^m R_f/(f^{\lambda_j}),
\qquad 1\le \lambda_j\le k-1,
$$
then
$$
\Hom_{R_f}(M,K)\cong K^m,
\qquad
\Ext^1_{R_f}(M,K)\cong K^m,
$$
and hence
$$
\#\Ext^1_{R_f}(M,K)=\#\Hom_{R_f}(M,K)=|K|^m.
$$
This proves the lemma.
\end{proof}

\subsection{Sharp function-field surjection moments}\label{subsec:sharp-function-field-moments}
We establish the function-field surjection-moment estimate at the additive critical window. As in Lee \cite[Lemma 2.1 and Sections 2.2-2.4]{lee2025sharp}, we reduce to two-point entry distributions and classify the Fourier sum by support sizes and popular values. We follow Lee's division into small and large supports, but our estimate treats his two smaller-support ranges together. The first two large-support cases follow essentially from Lee \cite[Propositions 2.6 and 2.7]{lee2025sharp}, whereas the fully orthogonal case requires a refinement of Lee \cite[Proposition 2.10]{lee2025sharp}. The key input is a convexity-based averaged Fourier damping estimate, which retains the full first-order decay in $n\alpha_n$ in both the small-support and fully orthogonal cases.

\begin{thm}[Sharp function-field surjection moments]\label{thm:function-field-moments}
Suppose that $n\alpha_n-\log n\to+\infty$ and $A_n\in\Mat_n(\F_p)$ has independent $\alpha_n$-balanced entries. Set $X_n:=\Cok_{\F_p[t]}(tI_n-A_n)$. Then for every fixed $\F_p[t]$-module $G$ with $m:=\dim_{\F_p}G<\infty$,
\begin{equation}\label{eq:function-field-moment-one}
    \lim_{n\to\infty}\E\#\Sur_{\F_p[t]}(X_n,G)=1.
\end{equation}
\end{thm}

\begin{prop}[Endpoint absolute Fourier estimate]\label{lem:absolute-Fourier-estimate}
Suppose that $n\alpha_n-\log n\to+\infty$. Let $W$ be a fixed $m$-dimensional $\F_p$-vector space, and let $\mu_{ij}$, $1\le i,j\le n$, be arbitrary $\alpha_n$-balanced probability measures on $\F_p$. Let $\zeta:=e^{2\pi i/p}$, and fix the additive character $\psi:\F_p\to\mathbb C^\times$ defined by $\psi(x):=\zeta^x$. For $1\le i,j\le n$, write
$$
    \widehat\mu_{ij}(u):=\sum_{a\in\F_p}\mu_{ij}(a)\psi(ua).
$$
Then
\begin{equation}\label{eq:absolute-Fourier-majorant}
    p^{-mn}\sum_{\substack{g_1,\ldots,g_n\in W\\\Span_{\F_p}\{g_1,\ldots,g_n\}=W}}\ \sum_{\substack{h_1,\ldots,h_n\in W^\vee\\(h_1,\ldots,h_n)\ne0}}\prod_{i,j=1}^n\left|\widehat\mu_{ij}\bigl(h_j(g_i)\bigr)\right|=o(1),
\end{equation}
uniformly over the arrays $(\mu_{ij})$.
\end{prop}

\begin{proof}[Proof of \Cref{thm:function-field-moments}, assuming \Cref{lem:absolute-Fourier-estimate}]
The assertion is immediate for the zero module. Let $G$ be a nonzero finite $\F_p[t]$-module, let $m:=\dim_{\F_p}G$, and let $T:G\to G$ denote multiplication by $t$. An $\F_p[t]$-linear map $\F_p[t]^n\to G$ is determined by the images $g_1,\ldots,g_n$ of the standard basis vectors. It factors through $X_n$ exactly when
\begin{equation}\label{eq:function-field-factor}
    Tg_j=\sum_{i=1}^n A_{ij}g_i,\qquad 1\le j\le n.
\end{equation}
If this condition holds, then $\Span_{\F_p}\{g_1,\ldots,g_n\}$ is stable under $T$. Hence the vectors generate $G$ as an $\F_p[t]$-module if and only if they span $G$ as an $\F_p$-vector space. Consequently,
\begin{equation}\label{eq:function-field-moment-formula}
    \E\#\Sur_{\F_p[t]}(X_n,G)=\sum_{\substack{g_1,\ldots,g_n\in G\\\Span_{\F_p}\{g_1,\ldots,g_n\}=G}}\P\left(Tg_j=\sum_{i=1}^n A_{ij}g_i,\ 1\le j\le n\right).
\end{equation}
Let $G^\vee:=\Hom_{\F_p}(G,\F_p)$ and let $\mu_{ij}$ denote the distribution of $A_{ij}$. Fourier inversion gives
\begin{equation}\label{eq:function-field-Fourier}
    \begin{aligned}
    \E\#\Sur_{\F_p[t]}(X_n,G)&=p^{-mn}\sum_{\substack{g_1,\ldots,g_n\in G\\\Span_{\F_p}\{g_1,\ldots,g_n\}=G}}\ \sum_{h_1,\ldots,h_n\in G^\vee}\\
    &\quad\times\psi\left(-\sum_{j=1}^n h_j(Tg_j)\right)\prod_{i,j=1}^n\widehat\mu_{ij}\bigl(h_j(g_i)\bigr).
    \end{aligned}
\end{equation}
The zero Fourier mode contributes
\begin{equation}\label{eq:zero-Fourier-mode}
    p^{-mn}\#\Sur_{\F_p}(\F_p^n,G)=\prod_{a=0}^{m-1}(1-p^{a-n})\longrightarrow1.
\end{equation}
For every other mode, the factor involving $T$ has absolute value one. Thus the absolute value of the remaining contribution is bounded by the left-hand side of \eqref{eq:absolute-Fourier-majorant}, which is $o(1)$ by \Cref{lem:absolute-Fourier-estimate}. This proves \eqref{eq:function-field-moment-one}.
\end{proof}

We now prove \Cref{lem:absolute-Fourier-estimate}. The case $m=0$ is immediate, so assume $m\ge1$ and identify $W$ and $W^\vee$ with $\F_p^m$ by dual bases. Decreasing the balance parameter only weakens the assumption, so we may assume $\alpha_n\le1/2$.

Following the extreme-point reduction of Lee \cite[Lemma 2.1]{lee2025sharp}, it suffices to consider entry distributions of the form $(1-\alpha_n)\delta_a+\alpha_n\delta_b$, where $a\ne b$: the left-hand side of \eqref{eq:absolute-Fourier-majorant} is separately convex in the entry distributions, and these two-point laws are the extreme points of the set of $\alpha_n$-balanced measures when $\alpha_n\le1/2$. Set
\begin{equation}\label{endpoint:eq:cdef}
    c_{\alpha_n}(x):=|1-\alpha_n+\alpha_n\zeta^x|,\qquad x\in\F_p,
\end{equation}
where $\zeta$ is as in \Cref{lem:absolute-Fourier-estimate}. For a two-point law, $|\widehat\mu_{ij}(g_i\cdot h_j)|=c_{\alpha_n}(u_{ij}(g_i\cdot h_j))$ with $u_{ij}\in\F_p^\times$. We henceforth bound the resulting Fourier sum uniformly over $(u_{ij})\in(\F_p^\times)^{n\times n}$.

\subsubsection{Support-size decomposition}
For tuples $g=(g_1,\ldots,g_n)$ and $h=(h_1,\ldots,h_n)$ in \eqref{eq:absolute-Fourier-majorant}, write $i:=|\{k:g_k\ne0\}|$ and $j:=|\{l:h_l\ne0\}|$ for their support sizes. Both are positive, since $g_1,\ldots,g_n$ span $\F_p^m$ and $h\ne0$. Fix a sufficiently small constant $\gamma_2>0$, independent of $n$. We divide the Fourier sum into two parts: the \emph{small-support range}, in which $i<\gamma_2 n$ or $j<\gamma_2 n$, and the \emph{large-support range}, in which $i,j\ge\gamma_2 n$. This uses the same small/large cutoff at $\gamma_2 n$ as Lee \cite[Section 2.2]{lee2025sharp}. Unlike Lee \cite[Theorems 2.2 and 2.3]{lee2025sharp}, we do not split the small-support range further, since one endpoint estimate handles it in full. For large supports we follow the popular-value classification of Lee \cite[Section 2.3]{lee2025sharp}.

We first establish an averaged Fourier damping estimate, then use it to handle all small-support terms. Finally, we treat the three large-support cases, sharpening Lee's two-block estimate in the fully orthogonal case.

\begin{lemma}[Averaged Fourier damping]\label{endpoint:lem:damping}
Let $U$ be a nonzero finite $\F_p$-vector space, and let $\ell_1,\ldots,\ell_s:U\to\F_p$ be nonzero linear forms, where $s\ge1$. Then
\begin{equation}\label{endpoint:eq:damping}
    \frac1{|U|}\sum_{v\in U}\prod_{a=1}^s c_{\alpha_n}(\ell_a(v))\le\frac{1+e^{-2\alpha_n(1-\alpha_n)s}}2.
\end{equation}
No independence assumption on the linear forms is needed.
\end{lemma}
\begin{proof}
Recall that $c_{\alpha_n}(x)=|1-\alpha_n+\alpha_n e^{2\pi ix/p}|$. By the law of cosines, we have
$$
    \begin{aligned}
    c_{\alpha_n}(x)^2
    &=(1-\alpha_n)^2+\alpha_n^2+2\alpha_n(1-\alpha_n)\cos(2\pi x/p)\\
    &=1-2\alpha_n(1-\alpha_n)(1-\cos(2\pi x/p))\\
    &=1-4\alpha_n(1-\alpha_n)\sin^2(\pi x/p).
    \end{aligned}
$$
Using $1-y\le e^{-y}$ for $y\ge0$, we obtain
$$
    c_{\alpha_n}(x)\le\exp\bigl(-2\alpha_n(1-\alpha_n)\sin^2(\pi x/p)\bigr).
$$
Since each $\ell_a$ is a nonzero linear form on $U$, it is surjective onto $\F_p$. Consequently,
$$
    \frac1{|U|}\sum_{v\in U}\sin^2(\pi\ell_a(v)/p)
    =\frac1p\sum_{x\in\F_p}\sin^2(\pi x/p)=\frac12.
$$
It follows that
$$
    \frac1{|U|}\sum_{v\in U}\sum_{a=1}^s\sin^2(\pi\ell_a(v)/p)=\frac{s}{2}.
$$
On the other hand, by convexity of the exponential function, for any $0\le y\le s$ we have
$$
    e^{-2\alpha_n(1-\alpha_n)y}
    \le 1-\frac{1-e^{-2\alpha_n(1-\alpha_n)s}}{s}y.
$$
Combining these observations, we obtain
$$
    \begin{aligned}
    \frac1{|U|}\sum_{v\in U}\prod_{a=1}^s c_{\alpha_n}(\ell_a(v))
    &\le\frac1{|U|}\sum_{v\in U}\exp\left(-2\alpha_n(1-\alpha_n)\sum_{a=1}^s\sin^2(\pi\ell_a(v)/p)\right)\\
    &\le\frac1{|U|}\sum_{v\in U}\left(1-\frac{1-e^{-2\alpha_n(1-\alpha_n)s}}{s}\sum_{a=1}^s\sin^2(\pi\ell_a(v)/p)\right)\\
    &=1-\frac{1-e^{-2\alpha_n(1-\alpha_n)s}}{s}\cdot\frac{s}{2}\\
    &=1-\frac{1-e^{-2\alpha_n(1-\alpha_n)s}}{2},
    \end{aligned}
$$
which completes the proof.
\end{proof}

In both the small-support range and the fully orthogonal large-support case, we need to estimate the decay provided by \Cref{endpoint:lem:damping} in terms of the number of exceptional coordinates. The following elementary inequalities give a sharp first-order bound when this number is small and a uniform lower bound when it is larger.

\begin{lemma}[Elementary damping bounds]\label{endpoint:lem:damping-bounds}
For every $s\ge0$ and $0<\alpha_n\le1/2$, we have
\begin{equation}\label{endpoint:eq:Dbounds}
    \alpha_n(1-\alpha_n)s-\alpha_n^2(1-\alpha_n)^2s^2
    \le\frac{1-e^{-2\alpha_n(1-\alpha_n)s}}2
    \le\alpha_ns.
\end{equation}
Moreover,
\begin{equation}\label{endpoint:eq:Dtworegimes}
    \frac{1-e^{-2\alpha_n(1-\alpha_n)s}}2
    \ge
    \begin{cases}
    (1-e^{-1})\alpha_n(1-\alpha_n)s,&2\alpha_n(1-\alpha_n)s\le1,\\
    (1-e^{-1})/2,&2\alpha_n(1-\alpha_n)s>1.
    \end{cases}
\end{equation}
\end{lemma}

\begin{proof}
For every $x\ge0$, the elementary Taylor bounds give
$$
    x-\frac{x^2}{2}\le1-e^{-x}\le x.
$$
Taking $x=2\alpha_n(1-\alpha_n)s$ and using $1-\alpha_n\le1$ proves \eqref{endpoint:eq:Dbounds}.

For the second inequality, observe that
$$
    \frac{1-e^{-x}}{x}=\int_0^1 e^{-tx}\,dt
$$
is decreasing on $(0,\infty)$. Therefore, if $0<x\le1$, then
$$
    1-e^{-x}\ge(1-e^{-1})x.
$$
On the other hand, if $x>1$, then $1-e^{-x}\ge1-e^{-1}$. Applying these bounds with $x=2\alpha_n(1-\alpha_n)s$ proves \eqref{endpoint:eq:Dtworegimes}. The case $s=0$ is immediate.
\end{proof}

\subsubsection{Small supports}

We first consider tuples for which at least one of $g$ and $h$ has support size smaller than $\gamma_2 n$. This range corresponds to the union of $I_1$ and $I_2$ in Lee \cite[Section 2.2]{lee2025sharp}. Our averaged Fourier damping estimate allows us to treat both ranges simultaneously, without the separate arguments of Lee \cite[Theorems 2.2 and 2.3]{lee2025sharp}.

To control the number of possible supports, we use the binary entropy function
$$
    H(x):=-x\log x-(1-x)\log(1-x),\qquad 0<x<1,
$$
where logarithms are natural. Recall that
$$
    \binom nk\le\exp\bigl(nH(k/n)\bigr),\qquad 0\le k\le n,
$$
with the convention $H(0)=H(1)=0$. This bound will allow us to compare the combinatorial cost of choosing exceptional coordinates with the Fourier decay.

\begin{lemma}[Small-support summation]\label{endpoint:lem:small}
For every sufficiently small fixed $\gamma_2>0$,
\begin{equation}\label{endpoint:eq:Sn}
    \sum_{s=1}^{\lfloor\gamma_2 n\rfloor}\binom ns(p^m-1)^s\exp\left(-\frac n2(1-e^{-2\alpha_n(1-\alpha_n)s})\right)=o(1).
\end{equation}
\end{lemma}

\begin{proof}
Since $\alpha_n(1-\alpha_n)$ is increasing on $(0,1/2]$, each summand in \eqref{endpoint:eq:Sn} is nonincreasing in $\alpha_n$. Thus we may replace $\alpha_n$ by $\min\{\alpha_n,(\log n+\sqrt{\log n})/n\}$, retaining the same notation. This preserves $n\alpha_n-\log n\to+\infty$ and ensures
$$
    \log n\le n\alpha_n\le\log n+\sqrt{\log n}
$$
for all sufficiently large $n$.

First suppose $1\le s\le n/(\log n)^2$. By \Cref{endpoint:lem:damping-bounds},
$$
    \begin{aligned}
    \frac n2(1-e^{-2\alpha_n(1-\alpha_n)s})
    &\ge n\alpha_n(1-\alpha_n)s-n\alpha_n^2(1-\alpha_n)^2s^2\\
    &\ge s(n\alpha_n-n\alpha_n^2-n\alpha_n^2s)\\
    &\ge s(n\alpha_n-O(1)),
    \end{aligned}
$$
where the last inequality follows from $n\alpha_n^2=o(1)$ and $n\alpha_n^2s=O(1)$. Consequently, the contribution from this range is at most
$$
    \sum_{s=1}^{\lfloor n/(\log n)^2\rfloor}\binom ns(p^m-1)^s e^{-s(n\alpha_n-O(1))}
    \le\exp\bigl(O_{p,m}(ne^{-n\alpha_n})\bigr)-1=o(1).
$$

Next suppose $n/(\log n)^2<s\le\gamma_2 n$. If $2\alpha_n(1-\alpha_n)s\le1$, then \Cref{endpoint:lem:damping-bounds} gives
$$
    \frac n2(1-e^{-2\alpha_n(1-\alpha_n)s})
    \ge(1-e^{-1})n\alpha_n(1-\alpha_n)s
    \ge\frac{1-e^{-1}}2s\log n.
$$
On the other hand,
$$
    \log\left(\binom ns(p^m-1)^s\right)
    \le s\log\frac{en(p^m-1)}s
    =O_{p,m}(s\log\log n).
$$
The Fourier decay therefore dominates the combinatorial factor, and the contribution from this range tends to zero.

Finally, if $2\alpha_n(1-\alpha_n)s>1$, then \Cref{endpoint:lem:damping-bounds} gives
$$
    \frac n2(1-e^{-2\alpha_n(1-\alpha_n)s})\ge\frac{1-e^{-1}}2n.
$$
Since $H$ is increasing on $(0,1/2)$, the total combinatorial factor is bounded by
$$
    \sum_{s=0}^{\lfloor\gamma_2 n\rfloor}\binom ns(p^m-1)^s
    \le(n+1)\exp\bigl(n[H(\gamma_2)+\gamma_2\log(p^m-1)]\bigr).
$$
Choosing $\gamma_2>0$ sufficiently small that
$$
    H(\gamma_2)+\gamma_2\log(p^m-1)<\frac{1-e^{-1}}4
$$
makes this contribution $o(1)$ as well, completing the proof.
\end{proof}

\begin{lemma}[One small support]\label{endpoint:prop:one-small}
For sufficiently small $\gamma_2>0$, the contribution to \eqref{eq:absolute-Fourier-majorant} from tuples with $i<\gamma_2 n$ or $j<\gamma_2 n$ is $o(1)$, uniformly over the two-point entry distributions.
\end{lemma}

\begin{proof}
Suppose $g$ has support of size $s\ge1$. Fix its support and its nonzero values, and drop the spanning condition. For each $j$, the $s$ linear forms $h\mapsto u_{ij}(g_i\cdot h)$, indexed by the nonzero coordinates of $g$, are nonzero on $\F_p^m$. By \Cref{endpoint:lem:damping}, summing over $h_j\in\F_p^m$ gives an average of at most
$$
    \frac{1+e^{-2\alpha_n(1-\alpha_n)s}}2.
$$
After summing over all supports and nonzero values of $g$, the normalized contribution is therefore bounded by
$$
    \sum_{s=1}^{\lfloor\gamma_2 n\rfloor}\binom ns(p^m-1)^s\left(\frac{1+e^{-2\alpha_n(1-\alpha_n)s}}2\right)^n
    \le\sum_{s=1}^{\lfloor\gamma_2 n\rfloor}\binom ns(p^m-1)^s\exp\left(-\frac n2(1-e^{-2\alpha_n(1-\alpha_n)s})\right).
$$
This is $o(1)$ by \Cref{endpoint:lem:small}. The case $j<\gamma_2 n$ follows by the same argument, interchanging the roles of $g$ and $h$.
\end{proof}

\subsubsection{Large supports and popular values}

We now consider the remaining range $i,j\ge\gamma_2 n$, where both tuples have large supports. Following Lee \cite[Section 2.3]{lee2025sharp}, we classify these tuples according to their popular nonzero values. As in Lee's argument, the first two cases will be handled by direct Fourier decay and counting estimates, while the fully orthogonal third case will require a sharper analysis.

Choose a fixed constant $\gamma_3>0$ satisfying
$$
    \gamma_3<\min\left\{\frac{\gamma_2}{p^m-1},\frac1{2p^m}\right\},
$$
which may be taken smaller below. A nonzero value $v\in\F_p^m$ is called \emph{popular} for $g$ if it occurs among $g_1,\ldots,g_n$ at least $\gamma_3 n$ times. Let $A$ denote the set of popular nonzero values of $g$, and define $B$ similarly for $h$. Since $i,j\ge\gamma_2 n$ and $(p^m-1)\gamma_3<\gamma_2$, both $A$ and $B$ are nonempty. Moreover, at most $\eta n$ coordinates of either tuple lie outside zero and its popular nonzero values, where $\eta:=(p^m-1)\gamma_3$.

Following Lee \cite[Section 2.3]{lee2025sharp}, we divide the large-support range into three cases:
\begin{enumerate}
\item\label{endpoint:item:I} $\mathcal C_1$: There exist $v\in A$ and $w\in B$ such that $v\cdot w\ne0$.
\item\label{endpoint:item:II} $\mathcal C_2$: All pairs in $A\times B$ are orthogonal, and $(|A|+1)(|B|+1)<p^m$.
\item\label{endpoint:item:III} $\mathcal C_3$: All pairs in $A\times B$ are orthogonal, and $(|A|+1)(|B|+1)=p^m$.
\end{enumerate}
To see that these cases are exhaustive, suppose that all elements of $A$ are orthogonal to all elements of $B$, and let $V:=\Span_{\F_p}(A)$. Then $A\subseteq V\setminus\{0\}$ and $B\subseteq V^\perp\setminus\{0\}$, so
\begin{equation}\label{endpoint:eq:cardinal}
    |A|+1\le|V|,\qquad |B|+1\le|V^\perp|,\qquad |V|\,|V^\perp|=p^m.
\end{equation}
In particular, $(|A|+1)(|B|+1)\le p^m$. In Case $\mathcal C_3$, equality holds throughout \eqref{endpoint:eq:cardinal}, which forces
$$
    A=V\setminus\{0\},\qquad B=V^\perp\setminus\{0\}.
$$
Since both $A$ and $B$ are nonempty, $V$ is a nonzero proper subspace of $\F_p^m$, uniquely determined by $A$.

We first treat Cases $\mathcal C_1$ and $\mathcal C_2$, following Lee \cite[Propositions 2.6 and 2.7]{lee2025sharp}. In Case $\mathcal C_1$, a nonorthogonal pair of popular values produces a large number of nontrivial Fourier factors. In Case $\mathcal C_2$, the strict inequality $(|A|+1)(|B|+1)<p^m$ gives a sufficient saving in the number of possible tuples.

\begin{lemma}[A nonorthogonal popular pair]\label{endpoint:prop:C1}
The contribution to \eqref{eq:absolute-Fourier-majorant} from Case $\mathcal C_1$ is $o(1)$, uniformly over the two-point entry distributions.
\end{lemma}

\begin{proof}
For $x\ne0$, the identity
$$
    c_{\alpha_n}(x)^2=1-4\alpha_n(1-\alpha_n)\sin^2(\pi x/p)
$$
and the assumption $\alpha_n\le1/2$ imply
$$
    c_{\alpha_n}(x)\le\exp\bigl(-\alpha_n\sin^2(\pi/p)\bigr).
$$
In Case $\mathcal C_1$, a nonorthogonal popular pair occurs at least $\gamma_3^2n^2$ times. Since every $u_{ij}$ is nonzero, all these pairs contribute nontrivial Fourier factors. There are at most $p^{2mn}$ pairs of tuples, so their normalized contribution is at most
$$
    \exp\bigl(mn\log p-\gamma_3^2\sin^2(\pi/p)n^2\alpha_n\bigr)=o(1),
$$
where we used $n\alpha_n\ge\log n$ for all sufficiently large $n$. This is the argument of Lee \cite[Proposition 2.6]{lee2025sharp}.
\end{proof}

\begin{lemma}[Counting words with rare exceptional symbols]\label{endpoint:lem:words}
Let $E\subseteq\F_p^m$ have $a\ge1$ elements, and let $0<\eta<1/2$. The number of words in $(\F_p^m)^n$ with at most $\eta n$ coordinates outside $E$ is at most
\begin{equation}\label{endpoint:eq:words}
    (n+1)a^n\exp\bigl(n[H(\eta)+\eta m\log p]\bigr).
\end{equation}
\end{lemma}

\begin{proof}
For exactly $k$ coordinates outside $E$, there are at most $\binom nk a^{n-k}p^{mk}$ possible words. Since $k\le\eta n<n/2$, the entropy bound gives $\binom nk\le e^{nH(\eta)}$, while $p^{mk}\le e^{\eta nm\log p}$. Summing over at most $n+1$ possible values of $k$ proves the claim.
\end{proof}

\begin{lemma}[An incomplete orthogonal popular pair]\label{endpoint:prop:C2}
For $\gamma_3>0$ sufficiently small in terms of $p$, $m$, and $\gamma_2$, the contribution from Case $\mathcal C_2$ is $o(1)$.
\end{lemma}

\begin{proof}
There are only finitely many pairs of subsets $(A,B)$ of $\F_p^m\setminus\{0\}$. Fix such a pair arising in Case $\mathcal C_2$. By the definition of popular values, at most $\eta n$ coordinates of $g$ lie outside $A\cup\{0\}$, and the same holds for $h$ with $B\cup\{0\}$. Dropping the spanning condition and bounding all Fourier factors by one, \Cref{endpoint:lem:words} shows that the normalized contribution from this pair is at most
$$
    (n+1)^2\left(\frac{(|A|+1)(|B|+1)}{p^m}\right)^n
    \exp\bigl(2n[H(\eta)+\eta m\log p]\bigr).
$$
Since $(|A|+1)(|B|+1)<p^m$, this is at most
$$
    (n+1)^2\exp\Bigl(n\bigl[\log(1-p^{-m})+2H(\eta)+2\eta m\log p\bigr]\Bigr).
$$
The exponent is negative for sufficiently small $\gamma_3$, since $\eta=(p^m-1)\gamma_3\to0$. Summing over the finitely many possible pairs $(A,B)$ gives $o(1)$, as in Lee \cite[Proposition 2.7]{lee2025sharp}.
\end{proof}

We finally consider Case $\mathcal C_3$, in which all popular values are orthogonal and $(|A|+1)(|B|+1)=p^m$. As observed above, there is a nonzero proper subspace $V\subset\F_p^m$ such that $A=V\setminus\{0\}$ and $B=V^\perp\setminus\{0\}$. Following the two-block decomposition of Lee \cite[Proposition 2.10]{lee2025sharp}, we use \Cref{endpoint:lem:damping} to control the coordinates lying outside these subspaces.

Put $X:=\{i:g_i\notin V\}$ and $Y:=\{j:h_j\notin V^\perp\}$, and write $s:=|X|$ and $t:=|Y|$. By the definition of popular values, $0\le s,t\le\eta n$, where $\eta=(p^m-1)\gamma_3$. Since $g_1,\ldots,g_n$ span $\F_p^m$ and $V$ is proper, we have $s\ge1$.

\begin{lemma}[Two-block Fourier estimate]\label{endpoint:lem:blocks}
For a fixed nonzero proper subspace $V\subset\F_p^m$ and fixed exceptional support sizes $s,t$, the contribution from tuples in Case $\mathcal C_3$ associated with $V$ is at most
\begin{equation}\label{endpoint:eq:blocks}
\binom ns\binom nt(p^m-1)^{s+t}\exp\Bigl(-\frac{n-t}{2}(1-e^{-2\alpha_n(1-\alpha_n)s})-\frac{n-s}{2}(1-e^{-2\alpha_n(1-\alpha_n)t})\Bigr).  
\end{equation}
The estimate remains valid when $s=0$ or $t=0$, with the corresponding term in the exponent equal to zero.
\end{lemma}
\begin{proof}
Fix the sets $X,Y$ and the values $g_i\notin V$ for $i\in X$ and $h_j\notin V^\perp$ for $j\in Y$. We drop the spanning and popularity restrictions on the remaining coordinates. The Fourier factors indexed by $X\times Y$ have modulus at most one, while those indexed by $([n]\setminus X)\times([n]\setminus Y)$ equal one. As in Lee \cite[Proposition 2.10]{lee2025sharp}, we estimate the other two rectangular blocks separately.

For each $j\notin Y$, the forms $h\mapsto u_{ij}(g_i\cdot h)$, $i\in X$, are nonzero on $V^\perp$, because $g_i\notin V$. By \Cref{endpoint:lem:damping}, summing over $h_j\in V^\perp$ gives at most
$$
    |V^\perp|\frac{1+e^{-2\alpha_n(1-\alpha_n)s}}2.
$$
Similarly, for each $i\notin X$, the forms $g\mapsto u_{ij}(g\cdot h_j)$, $j\in Y$, are nonzero on $V$, and the sum over $g_i\in V$ is at most
$$
    |V|\frac{1+e^{-2\alpha_n(1-\alpha_n)t}}2.
$$
When either collection of forms is empty, the corresponding average is one, and the same estimate holds.

There are $\binom ns\binom nt$ choices of $X,Y$, and at most $(p^m-|V|)^s(p^m-|V^\perp|)^t$ choices of the exceptional values. Since $|V||V^\perp|=p^m$, the remaining cardinality factor after normalization by $p^{mn}$ is
$$
    \begin{aligned}
    \frac{(p^m-|V|)^s(p^m-|V^\perp|)^t|V|^{n-s}|V^\perp|^{n-t}}{p^{mn}}
    &=\left(\frac{p^m}{|V|}-1\right)^s\left(\frac{p^m}{|V^\perp|}-1\right)^t\\
    &\le(p^m-1)^{s+t}.
    \end{aligned}
$$
The lemma follows from $1-x\le e^{-x}$.
\end{proof}

\begin{lemma}[Two-support summation]\label{endpoint:lem:doublesum}
For every sufficiently small fixed $\eta>0$, depending only on $p$ and $m$, the sum of the right-hand side of \eqref{endpoint:eq:blocks} over $0\le s,t\le\eta n$ with $s+t\ge1$ is $o(1)$.
\end{lemma}
\begin{proof}
The summands in \eqref{endpoint:eq:blocks} are nonincreasing in $\alpha_n\in(0,1/2]$. Thus, for this proof, we may replace $\alpha_n$ by $\min\{\alpha_n,(\log n+\sqrt{\log n})/n\}$ and retain the same notation. This preserves $n\alpha_n-\log n\to+\infty$ and ensures $\log n\le n\alpha_n\le\log n+\sqrt{\log n}$ for all sufficiently large $n$.

Write $r:=s+t$. Suppose first that $1\le r\le n/(\log n)^2$. By \Cref{endpoint:lem:damping-bounds}, the damping cost in \eqref{endpoint:eq:blocks} satisfies
$$
    \begin{aligned}
    \frac{n-t}{2}(1-e^{-2\alpha_n(1-\alpha_n)s})+\frac{n-s}{2}(1-e^{-2\alpha_n(1-\alpha_n)t})
    &\ge n\alpha_n(1-\alpha_n)r-n\alpha_n^2(s^2+t^2)-2\alpha_n st\\
    &\ge r(n\alpha_n-O(1)).
    \end{aligned}
$$
The second line holds because $s^2+t^2\le r^2$, $2st\le r^2/2$, and the truncation gives $n\alpha_n^2=o(1)$, $n\alpha_n^2r=O(1)$, and $\alpha_n r=o(1)$. By Vandermonde's identity, $\sum_{s+t=r}\binom ns\binom nt=\binom{2n}r$. Hence the sum over this range is at most
$$
\sum_{r=1}^{\lfloor n/(\log n)^2\rfloor}\binom{2n}r(p^m-1)^r e^{-r(n\alpha_n-O(1))}\le\exp\bigl(O_{p,m}(ne^{-n\alpha_n})\bigr)-1=o(1).
$$

Next suppose $n/(\log n)^2<r\le2\eta n$. Since
$$
    (1-e^{-2\alpha_n(1-\alpha_n)s})+(1-e^{-2\alpha_n(1-\alpha_n)t})\ge1-e^{-2\alpha_n(1-\alpha_n)r},
$$
and $s,t\le\eta n$, the damping cost in \eqref{endpoint:eq:blocks} is at least
$$
    \frac{(1-\eta)n}{2}(1-e^{-2\alpha_n(1-\alpha_n)r}).
$$
If $2\alpha_n(1-\alpha_n)r\le1$, then \Cref{endpoint:lem:damping-bounds} bounds this below by $cr\log n$ for some constant $c>0$, whereas
$$
    \log\left(\binom{2n}r(p^m-1)^r\right)\le r\log\frac{2en(p^m-1)}r=O_{p,m}(r\log\log n).
$$
Thus the contribution in this regime is $o(1)$. If $2\alpha_n(1-\alpha_n)r>1$, the damping cost is at least $(1-\eta)(1-e^{-1})n/2$. On the other hand,
$$
    \sum_{r\le2\eta n}\binom{2n}r(p^m-1)^r\le(2n+1)\exp\bigl(2nH(\eta)+2\eta n\log(p^m-1)\bigr).
$$
For sufficiently small $\eta>0$, this combinatorial factor is dominated by the exponential Fourier decay. This completes the proof.
\end{proof}

\begin{lemma}[Fully orthogonal popular values]\label{endpoint:prop:C3}
For $\gamma_3>0$ sufficiently small in terms of $p$, $m$, and $\gamma_2$, the contribution from Case $\mathcal C_3$ is $o(1)$.
\end{lemma}
\begin{proof}
If $m=1$, this case cannot occur, since there is no nonzero proper subspace of $\F_p^m$. Assume $m\ge2$. For each tuple pair $(g,h)$ in Case $\mathcal C_3$, the subspace $V=\Span_{\F_p}(A)$ is uniquely determined, and the exceptional support sizes satisfy $1\le s\le\eta n$ and $0\le t\le\eta n$.

Fix $V$. By \Cref{endpoint:lem:blocks}, the contribution from this subspace is bounded by the sum in \Cref{endpoint:lem:doublesum}, which is $o(1)$ provided $\eta=(p^m-1)\gamma_3$ is sufficiently small. The number of nonzero proper subspaces of $\F_p^m$ is finite and independent of $n$. Summing over these subspaces completes the proof.
\end{proof}

Finally, we are ready for our proof of \Cref{lem:absolute-Fourier-estimate}.

\begin{proof}[Proof of \Cref{lem:absolute-Fourier-estimate}]
By the extreme-point reduction established above, it suffices to consider the two-point entry distributions, uniformly over $(u_{ij})\in(\F_p^\times)^{n\times n}$.

Choose $\gamma_2>0$ sufficiently small for \Cref{endpoint:lem:small}, and then choose $\gamma_3>0$ sufficiently small to satisfy the popular-value condition and the estimates in \Cref{endpoint:prop:C2} and \Cref{endpoint:lem:doublesum}.

By \Cref{endpoint:prop:one-small}, the contribution from tuples with at least one small support is $o(1)$. In the remaining range, both supports are at least $\gamma_2 n$, and every pair of tuples belongs to exactly one of Cases $\mathcal C_1$, $\mathcal C_2$, and $\mathcal C_3$. Their contributions are $o(1)$ by \Cref{endpoint:prop:C1} , \Cref{endpoint:prop:C2}, and \Cref{endpoint:prop:C3}, respectively.

All estimates are uniform over $(u_{ij})$, so \eqref{eq:absolute-Fourier-majorant} follows.
\end{proof}

\subsection{Proof of the limiting rational canonical form}\label{subsec:limiting-rcf}

\begin{proof}[Proof of \Cref{thm:main}]
Fix distinct irreducible polynomials $F_{i_1},\ldots,F_{i_r}$ and partitions $\l^{(1)},\ldots,\l^{(r)}\in\Y$. Set
\begin{equation}\label{eq:target-module-M}
    M:=\bigoplus_{a=1}^r\bigoplus_{j\ge1}\F_p[t]/\bigl(F_{i_a}^{\l_j^{(a)}}\bigr).
\end{equation}
For each $a$, let $k_a:=1+\max\{\l_j^{(a)}:j\ge1\}$, taking $k_a=1$ when $\l^{(a)}$ is the zero partition, and set $q(t):=\prod_{a=1}^r F_{i_a}^{k_a}$. By \Cref{thm: rational canonical and smith normal form},
\begin{equation}\label{eq:truncated-event}
    X_n/q(t)X_n\cong M\quad\Longleftrightarrow\quad\l^{(F_{i_a})}(A_n)=\l^{(a)}\ \text{for all }1\le a\le r.
\end{equation}
Every surjection from $X_n$ onto a module annihilated by $q(t)$ factors through $X_n/q(t)X_n$. Applying \Cref{thm:function-field-moments} to any fixed finite $\F_p[t]/(q(t))$-module $G$, we obtain
$$
    \lim_{n\to\infty}\E\#\Sur_{\F_p[t]/(q(t))}(X_n/q(t)X_n,G)=1.
$$
The robustness theorem for surjection moments of Sawin and Wood \cite[Theorem 1.6]{sawin2022moment} identifies the limiting distribution of these finite modules. Their explicit formula \cite[Lemma 6.3]{sawin2022moment} gives
\begin{equation}\label{eq:from-moment-to-Sawin-Wood}
    \lim_{n\to\infty}\P(X_n/q(t)X_n\cong M)=\frac1{\#\Aut(M)}\prod_{a=1}^r\prod_{j\ge1}\left(1-\frac{\#\Ext^1_{\F_p[t]/(q(t))}(M,\F_p[t]/(F_{i_a}))}{\#\Hom_{\F_p[t]}(M,\F_p[t]/(F_{i_a}))}p^{-j\deg F_{i_a}}\right).
\end{equation}
By the Chinese remainder theorem, the $\Hom$ and $\Ext^1$ groups for $F_{i_a}$ depend only on the $F_{i_a}$-primary part of $M$. Every part of $\l^{(a)}$ is smaller than $k_a$, so \Cref{lem:ext-hom-correct} shows that the ratio in \eqref{eq:from-moment-to-Sawin-Wood} equals $1$. Moreover, \Cref{prop: decomposition and explicit expression of Aut} gives $\#\Aut(M)=\prod_{a=1}^r\#\Aut_{F_{i_a}}(\l^{(a)})$. Consequently, by \eqref{eq:truncated-event},
\begin{align}
\begin{split}
    \lim_{n\to\infty}\P\bigl(\l^{(F_{i_a})}(A_n)=\l^{(a)},\ 1\le a\le r\bigr)&=\prod_{a=1}^r\left[\frac1{\#\Aut_{F_{i_a}}(\l^{(a)})}\prod_{j\ge1}(1-p^{-j\deg F_{i_a}})\right]\\
    &=\prod_{a=1}^r\mu_{F_{i_a}}(\l^{(a)}).
\end{split}
\end{align}
This proves \Cref{thm:main}.
\end{proof}

\begin{proof}[Proof of \Cref{cor:multiplicity}]
Cohen and Lenstra's partition measure satisfies
$$
    \P_{\mu_{F_i}}(|\l|=j)=p^{-j\deg F_i}\prod_{k\ge j+1}(1-p^{-k\deg F_i}),\qquad j\ge0;
$$
see Cohen and Lenstra \cite[Example 5.9(ii)]{cohen2006heuristics}. Applying \Cref{thm:main} to the primary partitions gives the asserted joint convergence of their sizes.
\end{proof}

\subsection{Critical sparse examples}\label{subsec:critical_sparse}

We conclude with sparse obstructions at the additive critical window. Throughout this subsection, $C\in\mathbb R$ is a fixed absolute constant, independent of $n$ and $d$. We first refine Lee's degree-one example and then describe its analogue for an irreducible polynomial of degree $d$.

\emph{Degree one.} Following Lee \cite{lee2025sharp}, put $\alpha_n:=(\log n+C)/n$ and take the entries of $A_n$ independently with $\P(A_{ij}=0)=1-\alpha_n$ and $\P(A_{ij}=1)=\alpha_n$. For sufficiently large $n$, this is an $\alpha_n$-balanced model. A fixed row is zero with probability
\begin{equation}\label{eq:critical-linear-probability}
    (1-\alpha_n)^n=\frac{e^{-C}+o(1)}n.
\end{equation}
Since distinct rows are independent, the number of zero rows converges in distribution to $\operatorname{Pois}(e^{-C})$. Each zero row contributes an independent vector to the left kernel of $A_n$. Thus for every fixed $k\ge1$,
\begin{equation}\label{eq:critical-linear-tail}
    \liminf_{n\to\infty}\P(\operatorname{corank}A_n\ge k)\ge\P(\operatorname{Pois}(e^{-C})\ge k)\ge e^{-e^{-C}}\frac{e^{-Ck}}{k!}.
\end{equation}
On the other hand, the Cohen-Lenstra measure for the $t$-primary partition gives a corank tail of order $O_p(p^{-k^2})$. For a sufficiently large fixed $k$, this is smaller than the right-hand side of \eqref{eq:critical-linear-tail}. Therefore convergence to the universal distribution fails at $\alpha_n=(\log n+C)/n$ for every fixed $C\in\mathbb R$. When $C=0$, this coincides with the original critical example of Lee \cite{lee2025sharp}.

\emph{An irreducible polynomial of degree $d$.} Fix a monic irreducible polynomial $F\in\F_p[t]$ of degree $d$, and let $K:=\F_p[t]/(F)$ be the corresponding finite field of order $p^d$. Take a deterministic matrix $B_n\in\Mat_n(\F_p)$ which is block diagonal with $\lfloor n/d\rfloor$ copies of the companion matrix $C(F)$ and a final block of size less than $d$. Put
\begin{equation}\label{eq:degree-d-critical-model}
    \alpha_n:=\frac{\log n+C}{dn},\qquad \P(A_{ij}=(B_n)_{ij})=1-\alpha_n,\qquad\P(A_{ij}=(B_n)_{ij}+1)=\alpha_n,
\end{equation}
with independent entries. These entry distributions are $\alpha_n$-balanced for large $n$.

A prescribed companion block is untouched if all entries in its $d$ rows agree with $B_n$. This has probability
\begin{equation}\label{eq:degree-d-block-probability}
    (1-\alpha_n)^{dn}=\frac{e^{-C}+o(1)}n.
\end{equation}
The events for distinct blocks are independent. Hence, if $N_n$ is the number of untouched blocks, then
\begin{equation}\label{eq:degree-d-Poisson}
    N_n\sim\operatorname{Bin}\left(\left\lfloor\frac nd\right\rfloor,(1-\alpha_n)^{dn}\right)\Longrightarrow\operatorname{Pois}(e^{-C}/d).
\end{equation}
Let $T$ denote multiplication by $t$ on $K$ and choose an $\F_p$-linear isomorphism $\iota:\F_p^d\to K$ with $\iota C(F)=T\iota$. An untouched block produces a surjective map $P:\F_p^n\to K$, supported on that block, satisfying $PA_n=TP$. It therefore induces an $\F_p[t]$-linear surjection $X_n\to K$. Surjections from distinct untouched blocks are linearly independent over $K$, giving
\begin{equation}\label{eq:critical-Hom-rank}
    \dim_K\Hom_{\F_p[t]}(X_n,K)\ge N_n.
\end{equation}
As in the degree-one case, the positive Poisson tail in \eqref{eq:degree-d-Poisson} is incompatible with the Cohen-Lenstra $F$-primary limit. This is because the same argument applies with $\F_p$ replaced by $K$ and the corresponding $K$-linear surjection counts. Consequently, universality fails at $\alpha_n=(\log n+C)/(dn)$ for every fixed $C$. In particular, the coefficient $1/d$ in \Cref{conj:degree_threshold} cannot be lowered; the construction does not establish sufficiency above this threshold.

\bibliographystyle{plain}
\bibliography{references.bib}

\end{document}